\documentclass[10pt]{amsart}
\usepackage{amsmath,amssymb,latexsym,amsthm, amscd, mathrsfs}

\theoremstyle{definition}
\theoremstyle{plain}
\newtheorem{prop}[subsubsection]{Proposition}
\newtheorem{thm}[subsubsection]{Theorem}
\newtheorem{lem}[subsubsection]{Lemma}
\newtheorem{cor}[subsubsection]{Corollary}
\newtheorem{rem}[subsubsection]{Remark}

\newcommand{\mbf}{\mathbf}
\newcommand{\mrm}{\mathrm}

\newcommand{\A}{\mathbb A}
\newcommand{\B}{\mbf B}
\newcommand{\U}{\mbf U}
\newcommand{\AU}{_\mathbb A \mbf U}

\title[Canonical bases of Borcherds-Cartan type]
{Canonical Bases of Borcherds-Cartan type}
\author{ Yiqiang Li}
\address{Department of Mathematics\\ Yale University \\
10 Hillhouse Avenue \\P.O. Box 208283\\ New Haven, CT 06520\\ fax: 203-432-7316}
\email{yiqiang.li@yale.edu}
\author{Zongzhu Lin}
\address{Department of Mathematics\\ Kansas State University\\ 
Manhattan, KS  66506}
\email{zlin@math.ksu.edu}
\thanks{\today}
\subjclass{Primary 17B37, Secondary 16G20, 14L30}

\begin{document}

\begin{abstract}
We study the canonical basis for the negative part $\U^-$ of the quantum generalized 
Kac-Moody algebra associated to a symmetric Borcherds-Cartan matrix.
The algebras $\U^-$ associated to two different matrices 
satisfying certain conditions may coincide (\ref{newalgebra}).
We show that the canonical bases coincide 
provided that the algebras $\U^-$ coincide (Theorem ~\ref{C=D}).
We also answer partially a question by Lusztig in ~\cite{lusztig3} 
(Theorem ~\ref{theorem1}).

\end{abstract}

\maketitle

\section{Introduction}
Let $\U^-(C)$ be the negative part of the quantized enveloping algebra of the 
Kac-Moody Lie algebra associated to a generalized Cartan matrix $C$. 
In ~\cite{lusztig1, lusztig2, lusztig4}, Lusztig defined and studied the 
canonical basis $\B(C)$ for $\U^-(C)$. The canonical basis $\B(C)$ has many
remarkable properties such as positivity and integrality (\cite{lusztig4}).
The appearance of the canonical basis 
is one of the major achievements in Lie theory.
(Kashiwara developed his $crystal$ $base$ theory in
~\cite{Kashiwara}. It was shown by Grojnowski and Lusztig (\cite{GL}) that the
global crystal base by Kashiwara coincides with the canonical basis by Lusztig.)

There are two quite different approaches, algebraic and geometric,
to define the canonical basis $\B(C)$ of $\U^-(C)$. 
For the geometric approach, the algebra $\U^-(C)$ is realized as
an algebra $\mathcal K_Q$, which is contained in the direct sum of the Grothendieck groups
of the bounded derived categories of complexes of sheaves on  representation varieties of certain quiver $Q$ 
without loops (\cite{lusztig2}). 
The irreducible perverse sheaves appearing in $\mathcal K_Q$ give rise to the canonical
basis $\B(C)$ of $\U^-(C)$. 
The ``simplest'' irreducible perverse sheaves in $\mathcal K_Q$, 
which generate $\mathcal K_Q$, 
correspond to the Chevalley generators of $\U^-(C)$.

When the matrix $C$ is replaced by a Borcherds-Cartan matrix, the algebra $\U^-(C)$ becomes 
the negative part of the quantized enveloping algebra of a generalized Kac-Moody algebra 
(quantum generalized Kac-Moody algebra for short). In this more general setting,
the algebraic approach is carried out in ~\cite{JKK}, 
and the global crystal base $\mathbb B(C)$ for $\U^-(C)$ is defined.
The geometric approach is taken recently in ~\cite{KS}, and 
the canonical basis $\B(C)$ for $\U^-(C)$ is defined.
The two bases $\mathbb B(C)$ and $\B(C)$ coincide under certain condition on $C$. 
It is conjectured that this condition can be removed.

One notices that in this more general setting, the Borcherds-Cartan matrix $C$ is not 
determined by the algebra $\U^-(C)$ uniquely. 
Under certain conditions on two given Borcherds-Cartan matrices $C$ and $D$, 
the algebras $\U^-(C)$ and $\U^-(D)$ coincide (\ref{newalgebra}). 
It is not clear if $\B(C)=\B(D)$ or $\mathbb B(C)=\mathbb B(D)$ when $\U^-(C)=\U^-(D)$.

The first main result in this paper says that the canonical bases $\B(C)$ and $\B(D)$
coincide provided that the algebras $\U^-(C)$ and $\U^-(D)$ coincide 
for two symmetric Borcherds-Cartan matrices $C$ and $D$. 
The second main result
in this paper answers partially a question in ~\cite{lusztig3}.
This question asked if the algebra $\mathcal K_Q$ mentioned above is still generated
by the ``simplest'' irreducible perverse sheaves in $\mathcal K_Q$ when
the quiver $Q$ has loops. 
In this more general setting, the algebra $\mathcal K_Q$ is bigger than $\U^-(C)$.
We show that $\mathcal K_Q$ is generated by the simplest possible elements
in $\mathcal K_Q$ when each imaginary vertex has at least two loops (Theorem ~\ref{theorem1}).
In general, the subalgebra generated by these simplest possible elements 
in $\mathcal K_Q$ has a canonical-type basis, 
and a characterization of these subalgebras is also given (Theorem ~\ref{theorem1}, 
Proposition ~\ref{prop}).

To prove the first result, we investigate in detail  a factorization
of Lusztig's resolution of the representation varieties of a quiver
with possible loops. 
Such a factorization was first appeared in ~\cite{KS}.
This factorization resolves the imaginary parts of the representation
varieties first  and then resolves the real part.
In the case when the flags used in the resolution on the imaginary part 
are full flags, we are able to construct three based algebras from certain
classes of simple perverse sheaves on the representation varieties and
the intermediate varieties in the factorization. 
It turns out that the three algebras are isomorphic and the bases of
the three algebras are compatible with each other.
One of the algebras is the one constructed in ~\cite{KS}, which is
shown to be the geometric realization of the quantum generalized
Kac-Moody algebras of the Borcherds-Cartan matrix of the quiver.
The other two algebras are new. One of the two new algebras is independent of
the number of the loops on each imaginary vertex. 
This leads us to a proof of the first result. 
(See Section ~\ref{newalgebra}.)

In order to prove the second result, we show that the subalgebra
generated by the generators $L_{i,n}$ has a basis, which we call the
canonical basis for this subalgebra, consisting
of certain semisimple perverse sheaves. In the case when the number of
loops is at least two for each imaginary vertex, semisimplicity can be
strengthen to simplicity. In this case the subalgebra coincides with
Lusztig's algebra. Lusztig's question is then answered affirmatively.
To show that the semisimple perverse sheaves, as basis elements for
the subalgebra,  are generated by the
generators $L_{i,n}$, we build up a bridge between the
full-flag resolutions and the partial-flag resolutions. 
There is no obvious relation between the full-flag and partial-flag
resolutions. 
This was done by reducing to a certain resolution of the representation variety
associated to the subquiver of the original quiver consisting only
non-loop arrows. 
Then we investigate the compatibility of several
induction  functors defined on various varieties along the bridge. 
This leads us to the proof of the second result.
The proof of the second result is the difficult part of this paper.

The paper is organized as follows. 
In Section 2, 
we define the quantum generalized Kac-Moody algebra, and its negative
part.
In Section 3, we study the factorization of the resolutions of 
the representation varieties of a quiver with loops. 
Then we relate this factorization
with the resolution of the representation varieties of the subquiver
consisting of all non-loop arrows.
Section 4 is devoted to the construction of  classes of
(semi)simple perverse on various varieties  in Section 3.
In Section 5, we study Lusztig's induction functor, and its several
variations. We also show  the compatibility of the various induction
functors defined.
In Section 6, we constructed various algebras from the classes of
(semi)simple perverse sheaves defined in Section 5. The first result
is then proved.
In Section 7, we prove the second result, and give a characterization
for Lusztig's algebra.

{\bf Acknowledgement.} The authors thank Olivier Schiffmann for
stimulating discussion via email  about the arguments in \cite{KS}.
The research was partially supported by NSF grant
DMS-0200673. Theorem ~\ref{theorem1} was reported by
Z. Lin in University of Missouri at Columbia, Oct. 2005, and by
Y. Li in the AMS special session on Representations of Groups and
Algebras, Eugene, Oregon, Nov. 2005.
We are very grateful to the referee for the suggestions in improving the paper.

\section{Quantum generalized Kac-Moody algebras}

\subsection{Symmetric Borcherds-Cartan matrices}
\label{Borcherds}
Let $I$ be a countable set. A $symmetric$ $Borcherds$-$Cartan$ matrix is a
matrix $C=(c_{ij})_{i,j \in I}$ satisfying 
\begin{align*}
c_{ii} \in \{ 2, 0, -2, -4\cdots\} \; \mrm{and}\; c_{ij}=c_{ji} \in
\mathbb Z_{\leq 0} \quad \mbox{for} \; i\neq j \in I.
\end{align*}
We set $I^{+}=\{ i\in I | c_{ii}=2\}$ and $I^-=I - I^{+}$.

\subsection{Quantum generalized Kac-Moody algebras}
\label{quantum}
Let $\mathbb Q(v)$ be the field of rational functions. 
For any $m\leq n$, we set 
\begin{align*}
[n]=\frac{v^n-v^{-n}}{v-v^{-1}}, \;\mbox{and}\; [n]^!=[n] [n-1]\cdots[1].
\end{align*}

The $quantum$ $generalized$ $Kac$-$Moody$ $algebra$ $\mbf U=\mbf U(C)$
attached to the symmetric Borcherds-Cartan matrix 
$C=(c_{ij})_{i,j \in  I}$ 
is an associative $\mathbb Q(v)$-algebra with generators
\begin{align*}
E_i, F_i, K_i \; \mrm{and} \; K_i^{-1} \;\;\; (i\in I)
\end{align*}
and satisfying the following relations
\begin{align*}
&K_i K_i^{-1}=1, \quad  K_i^{-1} K_i =1,  \quad K_i K_j= K_j K_i,\\ 
&K_i E_j= v^{c_{ij}} E_j K_i,  \quad K_i F_j= v^{-c_{ij}} F_j K_i,\\
&E_i F_j-F_j E_i= \delta_{ij} \frac{K_i-K_i^{-1}}{v-v^{-1}},\\
&\sum_{p=0}^{1-c_{ij}} (-1)^p E_i^{(p)} E_j E_i^{(1-c_{ij}-p)} =0 \quad
\mbox{for any }\;i\in I^+,\; j\in I, i\neq j,\\
&\sum_{p=0}^{1-c_{ij}} (-1)^p F_i^{(p)} F_j F_i^{(1-c_{ij}-p)} =0 \quad
\mbox{for any} \; i\in I^+,\; j\in I, i\neq j, \\
&E_iE_j=E_jE_i, \quad F_iF_j=F_jF_i\quad \quad \mbox{if} \; c_{ij}=0.
\end{align*}
Here we use the notations $E_i^{(n)}=\frac{E_i^n}{[n]^!}$ and
$F_i^{(n)}=\frac{F_i^n}{[n]^!}$ for $i\in I^+$ and $n\in \mathbb N$.

Let $\mbf U^-$ be the $\mathbb Q(v)$-subalgebra of $\mbf U$ generated by
$F_i$ for $i\in I$. 
Set $\mathbb A=\mathbb Z[v,v^{-1}]$, the ring of Laurent polynomials. 
Let $\AU^-$ denote the $\A$-subalgebra of $\U^-$ generated by the elements
$F_i^{(n)}$ for $i\in I^+$, and  $F_i$ for $i\in I^-$.

\section{Varieties associated to quivers}

\subsection{Quivers.} 
\label{quivers}
A $quiver$ is a quadruple  $Q=(I, \Omega, s,t: \Omega \to I)$ where $I$ is the
vertex set, $\Omega$ is the arrow set, and $s(\omega)$ 
and $t(\omega)$ are the starting and terminating vertices for
$\omega\in \Omega$, respectively.   
In this paper, we assume that the quiver is $locally$ $finite$:
the vertex set $I$ is countable and  
the set $\{\omega\in \Omega\;|\; s(\omega), t(\omega)\in J\}$ is
finite for any finite subset $J\subseteq I$. 
 
A vertex $i\in I$ is called $imaginary$ if 
$s(\omega)=t(\omega)=i$ for some $\omega \in \Omega$.
Let $I^-$ be the set of all imaginary vertices in $I$. We set
$I^+=I\backslash I^-$. Vertices in $I^+$ will be called $real$ vertices.

Let $\Omega(i)$ be  the subset of $\Omega$ consisting of all
$\omega$ such that $s(\omega)=t(\omega)=i$. For any $i\in I$, let
\begin{equation*}
l_i=\# \Omega(i).
\end{equation*}
Note that when $i\in I^+$, $l_i=0$. 
we set
\[
\Omega^{-}=\cup_{i\in I^-} \Omega(i)
\quad
\mbox{and}
\quad
\Omega^{+}=\Omega \backslash \Omega^-.
\]
(In other words, $\Omega^-$ consists of all loop arrows, while
$\Omega^+$ consists of all loop free arrows.)

The (nonsymmetric) $Euler$ $f\!orm$ of $Q$
$<,>: \mathbb Z[I] \times \mathbb Z[I] \to \mathbb Z$
is defined by
\begin{equation*}
<\alpha, \beta>= \sum_{i\in I} \alpha_i\beta_i -\sum_{\omega\in
  \Omega} \alpha_{s(\omega)} \beta_{t(\omega)}
\quad \mbox{for}\;\alpha, \beta \in \mathbb Z[I].
\end{equation*}
The $symmetric$ $Euler$ $f\!orm$
$(\, ,\,): \mathbb Z[I]\times \mathbb Z[I] \to \mathbb Z[I]$
is defined by  
\[(\alpha, \beta)=<\alpha, \beta> + <\beta,\alpha> 
\quad \mbox{for all} \; \alpha, \beta \in\mathbb Z[I].\]
Let
\[
c_{ij}=(i,j)\quad \mbox{for} \; i,j\in I.
\]
Then the matrix $C(Q)=(c_{ij})_{i,j\in I}$ is a  symmetric Borcherds-Cartan
matrix (see ~\ref{Borcherds}). We call $C=C(Q)$ the
Borcherds-Cartan matrix associated to $Q$.

\subsection{Representation  varieties associated to quivers}

We fix an algebraically closed field $k$.
Given any $I$-graded $k$-vector space $V=\oplus_{i\in I} V_i$, let 
\begin{equation*}
G_V=\Pi_{i\in I}\mrm{GL}(V_i)
\end{equation*}
where $\mrm{GL}(V_i)$ is the general linear group of $V_i$.  Let 
\[
E_V=\oplus_{\omega\in \Omega}\mrm{Hom}_k(V_{s(\omega)},V_{t(\omega)}).\] 
$G_V$ acts on $E_V$ by conjugation:
\[
(g.x)_{\omega}=g_{t(\omega)} x_{\omega} g_{s(\omega)}^{-1}
\]
for any $g\in G_V$ and $x\in E_V$.
Given any subset $\Delta \subseteq \Omega$, 
let 
\[
E_{V,\Delta}=\oplus_{\omega \in \Delta} \mrm{Hom}
(V_{s(\omega)},V_{t(\omega)}).
\] 
In particular,
$E_{V,\Omega^+}=\oplus_{\omega \in \Omega^+} 
\mrm{Hom}(V_{s(\omega)},  V_{t(\omega)}).
$

\subsection{Flag varieties}
\label{flag}
Given any $\nu \in \mathbb N[I]$, let $\mathcal X_{\nu}$ be the set of
all pairs $(\mbf i, \mbf a)$ where
\[
\mbf i=(i_1,\cdots,i_n)  \; \;(i_l\in I) 
\quad \mbox{and} \quad 
\mbf a=(a_1,\cdots,a_n) \;\; (a_l\in \mathbb N)
\]
such that $a_1 i_1 +\cdots+ a_ni_n=\nu$.
Let $V$ be an $I$-graded vector space of dimension $\nu$, a flag 
\[
\mbf F^{\bullet}=(V=V^0\supseteq V^1 \supseteq \cdots \supseteq V^n=0)
\]
is called $o\!f$ $type$ $(\mbf i,\mbf a)$ ($\in \mathcal X_{\nu}$) if
$\dim V^{l-1}/V^l=a_li_l$ for $l=1,\cdots, n$. We set
\[
\mathcal F_{\mbf i,\mbf a}=\{\mbox{all flags}\; \mbf F^{\bullet}
 \;\mbox{of type}\; (\mbf i,\mbf a)\}.
\]
Note that $\mathcal F_{\mbf i,\mbf a}$ is a smooth, irreducible
projective variety. (Indeed, it is a product of the partial flag
varieties.)

$G_V$ acts on $\mathcal F_{\mbf{i,a}}$ transitively, i.e.,
\[
g.F^{\bullet}=(V=V^0\supseteq g(V^1) \supseteq g(V^2) \supseteq \cdots \supseteq V^n=0)
\]
for any $g\in G_V$ and $\mbf F^{\bullet}\in \mathcal F_{\mbf{i,a}}$.

\subsection{(Partial) Resolutions of representation varieties}
\label{resolution}
For any $x\in E_V$ and $\mbf F^{\bullet}\in \mathcal F_{\mbf{i,a}}$,
$\mbf F^{\bullet}$ is $x$-$stable$ if
$
x_{\omega}(V^{l-1}_{s(\omega)}) \subseteq V^{l}_{t(\omega)}$
for any $\omega\in \Omega$ and $l=1, \cdots, n$.
Let
\[
\widetilde{\mathcal F}_{\mbf{i,a}}=\{ (x, \mbf F^{\bullet}) \in
E_V\times \mathcal F_{\mbf{i,a}} \;|\; \mbf F^{\bullet} \; \mbox{is} \;
x \mbox{-stable}\}.
\]
$G_V$ acts on $\widetilde{\mathcal F}_{\mbf{i,a}}$ by 
$g.(x, \mbf F^{\bullet})=(g.x, g.\mbf F^{\bullet})$ 
for any $g\in G_V$ and
$(x,\mbf F^{\bullet})\in \widetilde{\mathcal F}_{\mbf{i,a}}$.
We have the following diagram
\[
\begin{CD}
\widetilde{\mathcal F}_{\mbf{i,a}} @>\pi_{\mbf{i,a}}>> E_V\\
@Vp_{\mbf{i,a}}VV\\
\mathcal F_{\mbf{i,a}} 
\end{CD}
\]
where $\pi_{\mbf{i,a}}$ and $p_{\mbf{i,a}}$  are the first and 
second projections, respectively.
From the above diagram, one deduces the following well-known facts
\begin{lem}[\cite{lusztig2}]
\label{smooth}
\begin{enumerate}
\item $\pi_{\mbf{i,a}}$ is a $G_V$-equivariant, projective morphism.
\item $p_{\mbf{i,a}}$ is a vector bundle.
\item $\widetilde{\mathcal F}_{\mbf{i,a}}$ is a smooth, irreducible variety.
\end{enumerate}
\end{lem}

\subsection{A factorization of $\pi_{\mbf{i,a}}$}
\label{factorization}
For any $\nu \in \mathbb N[I]$, we set $\nu^-=\sum_{i\in I^-} \nu_i i$,
the component of $\nu$ supported on $I^-$. (Recall that $I^-$ is  the set of all imaginary
vertices.) For any pair $(\mbf{i,a})\in \mathcal X_{\nu}$, let
$(\mbf{i^-,a^-}) \in \mathcal X_{\nu^-}$ be the pair obtained from 
$(\mbf{i,a})$ by deleting the real vertices in $\mbf i$ and the
corresponding entries in $\mbf a$. Let
\[
\mathcal E_{\mbf{i,a}}=\{ (x, \mbf D^{\bullet}) \in E_V\times
\mathcal F_{(\mbf{i^-,a^-})}  |\; x_{\omega} (D_{s(\omega)}^{l-1}) \subseteq
D_{t(\omega)}^l,
\;
\forall\; \omega\in \Omega^- 
\;
\mbox{and}
\;
1\leq l \leq n
\}.
\]

Suppose that $V$ is an $I$-graded vector space of dimension $\nu$, we
set $V^-=\oplus_{i\in I^-} V_i$, the imaginary component of $V$.
For any flag $\mbf F^{\bullet}\in \mathcal F_{\mbf{i,a}}$,  
we set $D^l=V^l\cap V^-$. Then the flag
\[
\mbf F^{\bullet} \cap V^-=
(V^-=D^0 \supseteq D^1 \supseteq D^2\supseteq \cdots
\supseteq 0)
\]
can be regarded as a flag in $\mathcal F_{(\mbf{i^-,a^-})}$.
We have the following diagram
\[
\begin{CD}
\widetilde{\mathcal F}_{\mbf{i,a}} @>\mbf{\phi_{i,a}}>> 
\mathcal E_{\mbf{i,a}} @>\mbf{\psi_{i,a}}>> E_V,
\end{CD}
\]
where $\phi_{\mbf{i,a}}: (x, \mbf F^{\bullet}) \mapsto (x, \mbf
F^{\bullet}\cap V^-)$ and $\psi_{\mbf{i,a}}: (x, \mbf D^{\bullet})
\mapsto x$ are natural projections. By definition, we have
\begin{align*}
\pi_{\mbf{i,a}}&=\psi_{\mbf{i,a}} \phi_{\mbf{i,a}}.
\end{align*}

Note that $\phi_{\mbf{i,a}}$ is proper and $G_V$-equivariant, and
$\psi_{\mbf{i,a}}$ is semismall and $G_V$-equivariant. Further, when
$l_i\geq 2$ for all $i\in I^-$, $\psi_{\mbf{i,a}}$ is a small resolution.   
For a proof of these facts, see ~\cite{lusztig3},  ~\cite{KS} and ~\cite{KSaddendum}.

\subsection{Relation with varieties associated to quivers without loops}
\label{relation}
Similar to $\widetilde{\mathcal F}_{\mbf{i,a}}$ and 
$\mathcal E_{\mbf{i,a}}$, we define
\begin{align*}
\widetilde{\mathcal F}_{\mbf{i,a}}^+&
= \{ (x, \mbf F^{\bullet}) \in E_{V, \Omega^{+}}\times \mathcal
F_{\mbf{i,a}}\; | \; \mbf F^{\bullet} \; \mbox{is}\;
x\mbox{-stable}\}, \; \mbox{and}\\
\mathcal E_{\mbf{i,a}}^+&=E_{V,\Omega^+}\times \mathcal F_{\mbf{i^-,a^-}}.
\end{align*}
Combining with the diagram in Section ~\ref{factorization}, we have
the diagram
\[
\begin{CD}
\widetilde{\mathcal F}_{\mbf{i,a}} @>\phi_{\mbf{i,a}}>> \mathcal
E_{\mbf{i,a}} @>\psi_{\mbf{i,a}}>> E_V\\
@V\gamma_{\mbf{i,a}}VV @V\delta_{\mbf{i,a}}VV @VVV\\
\widetilde{\mathcal F}_{\mbf{i,a}}^+ @>\phi_{\mbf{i,a}}^+>> \mathcal
E_{\mbf{i,a}}^+ @>\psi_{\mbf{i,a}}^+>> E_{V,\Omega^+},
\end{CD}
\]
where the first row is the diagram in Section ~\ref{factorization},
the second row is defined in a similar manner as the first row, and the vertical maps are the
self-explained projections. Observe that
\begin{enumerate}
\item[(1)] $\gamma_{\mbf{i,a}}$ and $\delta_{\mbf{i,a}}$ are vector
      bundles.
\item[(2)] The left square is cartesian.
\end{enumerate}

\subsection{A bridge}
\label{bridge}
For any pair 
$(\mbf{i,a})=((i_1,\cdots, i_n) ,(a_1,\cdots,a_n))\in \mathcal X_{\nu}$ 
with  the imaginary vertex $i_{j_1}, i_{j_2},\cdots i_{j_s}$
($j_1<\cdots <j_s$) in $\mbf i$, 
we set
\[
\mbf j=
(i_1,\cdots, i_{j_1-1},  
\overbrace{i_{j_1},\cdots, i_{j_1}}^{a_{j_1}\; \mbox{copies}},
i_{j_1+1},\cdots, i_{j_2-1}, 
\cdots, 
\overbrace{i_{j_s},\cdots, i_{j_s}}^{a_{j_s}\; \mbox{copies}}, 
i_{j_s+1}, \cdots, i_n).
\]
If we set $\mbf b$ to be the sequence of the same length as $\mbf j$
and with all entries corresponding to the imaginary vertex in $\mbf b$
are 1 and all entries corresponding to the real vertex in $\mbf b$
equal to $a_l$, 
then the pair $(\mbf{j,b}) \in \mathcal X_{\nu}$. 
By abusing the notation, We write $\mbf j$ for $(\mbf{j,b})$. 
Consider the following diagram
\[
\begin{CD}
\widetilde{\mathcal F}_{\mbf j}^+ @>\phi_{\mbf j}^+>> \mathcal
E_{\mbf j}^+\\
@V\alpha_{\mbf{i,a}}VV  @V\beta_{\mbf{i,a}}VV\\
\widetilde{\mathcal F}_{\mbf{i,a}}^+ @>\phi_{\mbf{i,a}}^+>> \mathcal
E_{\mbf{i,a}}^+
\end{CD}
\]
where
the rows are defined in Section ~\ref{factorization} and the vertical
maps are obvious projections. Observe that
\begin{enumerate}
\item[(1)] $\alpha_{\mbf{i,a}}$ and $\beta_{\mbf{i,a}}$ are smooth with
      connected fibres. The square is cartesian.
\end{enumerate}

\section{Perverse sheaves on varieties associated to quivers}

\subsection{Notations.}
We fix some notations, most of them are taken from  ~\cite{lusztig5}.

Fix a prime $l$ that is invertible in $k$.
Given any algebraic variety $X$ over $k$, 
denote by $\mathcal{D}(X)$ the bounded derived category of complexes 
of $l$-adic sheaves on $X$ (\cite{BBD}).
Let $\mathcal M(X)$ be the full subcategory of $\mathcal D(X)$
consisting of all perverse sheaves on $X$ (\cite{BBD}).

Let $\bar{\mathbb Q}_l$ be an algebraic closure of the field of $l$-adic numbers.
By abuse of notation, denote by  $\bar{\mathbb Q}_l=(\bar{\mathbb Q}_l)_X$ 
the complex concentrated on degree zero, corresponding to
the constant $l$-adic sheaf over $X$. 
For any complex $K \in \mathcal D(X)$ and $n\in \mathbb Z$, 
let $K[n]$ be the complex such that $K[n]^i=K^{n+i}$ and the
differential is multiplied by a factor $(-1)^n$.
Denote by $\mathcal M(X)[n]$ the full subcategory of 
$\mathcal D(X)$ whose objects are of the form $K[n]$ with 
$K\in \mathcal M(X)$.
For any $K\in \mathcal D(X)$ and $L\in \mathcal D(Y)$, denote
by $K\boxtimes L$ the external tensor product of $K$ and $L$ in
$\mathcal D(X\times Y)$.

Let $f: X\to Y$ be a morphism of varieties, denote by
$f^*: \mathcal D(Y) \to \mathcal D(X)$ and 
$f_!: \mathcal D(X) \to \mathcal D(Y)$ 
the inverse image functor and the direct image functor with compact support, respectively.

Let $G$ be a connected algebraic group. Assume that $G$ acts on $X$
algebraically. Denote by $\mathcal D_G(X)$ the full subcategory of
$\mathcal D(X)$ consisting of all $G$-equivariant complexes over
$X$.
Similarly, denote by $\mathcal M_G(X)$ the full subcategory of
$\mathcal M(X)$ consisting of all $G$-equivariant perverse sheaves (\cite{lusztig5}).
If $G$ acts on $X$ algebraically and $f$ is a principal $G$-bundle,
then $f^*$ induces a functor (still denote by $f^*$) of equivalence
between $\mathcal M(Y)[\dim G]$ and $\mathcal M_G(X)$.
Its inverse functor is denoted by $f_{\flat}: \mathcal M_G(X) \to
\mathcal M(Y)[\dim G]$ (\cite{lusztig5}).

\subsection{A class of simple perverse sheaves on $E_V$}
\label{simple}
The complex 
$\bar{\mathbb Q}_l[\dim \widetilde{\mathcal F}_{\mbf{i,a}}]$ on 
$\widetilde{\mathcal F}_{\mbf{i,a}}$ is a simple $G_V$-equivariant 
perverse sheaf due to Lemma ~\ref{smooth} (3). 
By Lemma ~\ref{smooth} (1) and the Decomposition
theorem in \cite{BBD}, the complex
\[
L_{\mbf{i,a}}= (\pi_{\mbf{i,a}})_!(\bar{\mathbb Q}_l[\dim
\widetilde{\mathcal F}_{\mbf{i,a}}])
\]
is a $G_V$-equivariant semisimple complex on $E_V$.

Let 
$\mathcal P_V$ be the set of isomorphism classes of all simple
perverse sheaves $P$ appearing  in $L_{\mbf{i,a}}$ as a summand with a
possible shift for all $(\mbf{i,a}) \in \mathcal X_{\nu}$.

Let $\mathcal Q_V$ be the full subcategory of $\mathcal D(E_V)$
whose objects are finite direct sums of shifts of simple perverse
sheaves coming from $\mathcal P_V$. Note that all complexes in
$\mathcal Q_{V}$ are semisimple and $G_V$-equivariant.

Let $\mathcal Q_T \boxtimes \mathcal Q_W$ be the full subcategory of 
$\mathcal D(E_T\times E_W)$ whose objects are of the form
$P'\boxtimes P''$ for any $P'\in \mathcal Q_T$ and $P''\in \mathcal Q_W$.

\subsection{A class of $semisimple$ perverse sheaves on $E_V$}
\label{semisimple} 
Similar to the complex $L_{\mbf{i,a}}$ in Section ~\ref{simple}, we
see  that the complex
\[
L^-_{\mbf{i,a}}=(\phi_{\mbf{i,a}})_! (\bar{\mathbb Q}_l [\dim 
\widetilde{\mathcal F}_{\mbf{i,a}}])
\]
is a $G_V$-equivariant semisimple complex on $\mathcal E_{\mbf{i,a}}$
(see Section ~\ref{factorization} for notations).

Let $\widehat{\mathcal E}_{\mbf{i,a}}$ be the set of isomorphism
classes of all simple perverse sheaves on $\mathcal E_{\mbf{i,a}}$ 
appearing as direct summands in
the complex  $L^-_{\mbf{i,a}}$ with possible shifts.

\begin{lem} [\cite{KS}, ~\cite{KSaddendum}]
\label{perverse}
\begin{enumerate}
\item $(\psi_{\mbf{i,a}})_!(R)$ is a semisimple perverse sheaf on $E_V$  for any
      $R\in  \widehat{\mathcal E}_{\mbf{i,a}}$.
\item $(\psi_{\mbf{i,a}})_!(R)$ is  simple  if
      $l_i \geq 2$ for all $i\in I^-$.
\end{enumerate}
\end{lem}

It is proved in ~\cite{lusztig3} when $Q$ has one vertex and multiple loops.
When the entries in $\mbf a$ are $1$, 
the Lemma is proved in ~\cite[Proposition 4.1]{KS} and ~\cite[Section 1]{KSaddendum}. For arbitrary 
$\mbf a$, one can repeat the proof in ~\cite{KS} and ~\cite{KSaddendum} essentially word by
word.

Let $\mathcal R_V$ be the set of isomorphism classes of the
$semisimple$ perverse sheaves of the form: 
$(\psi_{\mbf{i,a}})_!(R)$ where 
$R\in \widehat{\mathcal  E}_{\mbf{i,a}}$ 
and $(\mbf{i,a}) \in \mathcal X_{\nu}$.
 
Let $\mathcal S_V$ be the full subcategory of $\mathcal D(E_V)$ whose
objects are finite direct sums of shifts of the semisimple perverse sheaves
coming from $\mathcal R_V$.

Similar to Section ~\ref{simple}, let
$\mathcal S_T \boxtimes \mathcal S_W$ be the full subcategory of
$\mathcal D(E_T\times E_W)$ whose objects are of the form 
$S'\boxtimes S''$ for any $S'\in \mathcal S_T$ and $S'' \in \mathcal S_W$.
Observe that 
\[
L_{\mbf{i,a}}=(\psi_{\mbf{i,a}})_!(L^-_{\mbf{i,a}}),
\]
one sees that $(\psi_{\mbf{i,a}})_!(R)$ is a direct sum of simple
perverse sheaves in $\mathcal P_V$.
Thus $\mathcal S_V$ (resp. $\mathcal S_T \boxtimes \mathcal S_W$) is a
full subcategory of $\mathcal Q_V$
(resp. $\mathcal Q_T \boxtimes \mathcal  Q_W$).

\subsection 
{Certain classes of simple perverse sheaves 
on $\mathcal  E_{\mbf{i,a}}$  and $\mathcal E_{\mbf{i,a}}^+$}
\label{certain}

Let $\mathcal Y_{\nu}$ be the  subset of $\mathcal X_{\nu}$ (\ref{flag})
consisting of all pairs $(\mbf{i,a})$ such that the (imaginary) entries in $\mbf a$ are 1.
For simplification, we write $\mbf i$ for $(\mbf{i,a})$.

For any $\mbf i\in \mathcal Y_{\nu}$, 
the variety $\mathcal E_{\mbf  i}$ can be identified with the variety
\[
\mathcal E_{\nu}=\{(x,\Pi_{i\in I^-} D_i^{\bullet}) \in E_V\times
\Pi_{i\in I^-}\mathcal F_{(i,\cdots,i)} \; |\; x_{\omega}(D_i^{l-1}) \subseteq D_i^l\;
\forall \omega\in \Omega(i), 1\leq l\leq \nu_i\} 
\] 
where $\mathcal F_{(i,\cdots,i)}$ is the full flag variety of $V_i$, 
$\dim V_i=\nu_i$ and $\Omega(i)$ is defined in Section ~\ref{quivers}.
(Note that  the varieties $\mathcal E_{\mbf{i,a}}$ in Section
~\ref{factorization} depend on the pair $\mbf{(i,a)} \in \mathcal
X_{\nu}$ in general.)

Now the diagram in Section ~\ref{relation} reads (for $\mbf i\in
\mathcal Y_{\nu}$)
\[
\begin{CD}
\widetilde{\mathcal F}_{\mbf i} @>\phi_{\mbf i}>> \mathcal E_{\nu} @>\psi_{\nu}>> E_V\\
@V\gamma_{\mbf i}VV @V\delta_{\nu}VV @V\epsilon VV\\
\widetilde{\mathcal F}_{\mbf i}^+ @>\phi_{\mbf i}^+>> 
\mathcal E_{\nu}^+ @>\psi_{\nu}^+>> E_{V,\Omega^+},
\end{CD}
\]
where $\mathcal E^+_{\nu}=E_{V,\Omega^+}\times \Pi_{i\in I^-} \mathcal F_{(i,\cdots,i)}$,
$\psi_{\nu}$ and $\psi_{\nu}^+$ are first projections.

Similar to $\widehat{\mathcal E}_{\mbf{i,a}}$, 
$\mathcal R_V$ and $\mathcal S_V$ defined in Section
~\ref{semisimple}, we define the following objects.

$\widehat{\mathscr E}_{\nu}$
is the set of isomorphism
classes of all simple perverse sheaves on $\mathcal E_{\nu}$ 
appearing as direct summands in
the complex  $L^-_{\mbf i}$ with possible shifts for all 
$\mbf i\in \mathcal Y_{\nu}$.

$\mathscr R_V$ is  the set of isomorphism classes of the
semisimple perverse sheaves of the form: 
$(\psi_{\nu})_!(R)$ for all
$R\in \widehat{\mathscr  E}_{\nu}$.
 
$\mathscr S_V$ be the full subcategory of $\mathcal D(E_V)$ whose
objects are finite direct sums of shifts of the semisimple perverse sheaves
coming from $\mathscr R_V$.

Finally, we define the following objects.

Let $\widehat{\mathscr F}_{\nu}$ be the 
full subcategory of $\mathcal D(\mathcal E_{\nu})$ whose
objects are finite direct sums of the simple perverse sheaves
coming from $\widehat{\mathscr E}_{\nu}$.

Let 
$\widehat{\mathscr E}_{\nu}^+$
be the set of isomorphism
classes of all simple perverse sheaves on $\mathcal E_{\nu}^+$ 
appearing as direct summands in
the complex  $(\phi_{\mbf i}^+)_!(\bar{\mathbb Q}_l)$ with possible shifts for all 
$\mbf i\in \mathcal Y_V$.

Let $\widehat{\mathscr F}_{\nu}^+$ be the 
full subcategory of $\mathcal D(\mathcal E_{\nu}^+)$ whose
objects are finite direct sums of shifts of  the semisimple perverse sheaves
coming from $\widehat{\mathscr E}_{\nu}^+$.

These objects will be used in Section ~\ref{newalgebra} to construct
various algebras.

\section{Induction functors}
\label{induction}

\subsection{Induction functors on $E_V$}
\label{inductionA}
Let $W \subseteq V$ be an $I$-graded subspace. 
Let $T=V/W$ and $p:V\to T$ the natural projection. We sometimes write
$p_T$ to avoid confusion.

Given any $x\in E_V$,
we write $x(W)\subseteq W$  if 
$x_{\omega}(W_{s(\omega)})\subseteq W_{t(\omega)}$, for all $\omega\in \Omega$.

If $x(W)\subseteq W$, it induces two elements $x_W$ and $x_T$ in $E_W$ and
$E_T$ respectively as follows. 
$(x_W)_{\omega}$ is the restriction of $x_{\omega}$ to $W$ 
for all $\omega\in \Omega$. 
$(x_T)_{\omega}$ is defined such that
$p_{t(\omega)}\, x_{\omega}=(x_T)_{\omega}\, p_{s(\omega)}$ for all $\omega\in \Omega$.

We consider the following diagram
\[
\label{ind}
\begin{CD}
E_T\times E_W @<q_1<< E' @>q_2>> E'' @>q_3>> E_V
\end{CD}
\]
where
$E''= \{ (x,V')\; |\; x(V')\subseteq V', \dim V'=\dim W\}$, and 
$E'$ is the variety consisting of all quadruples $ (x, V', r', r'')$ such that
$(x,V') \in E''$,  $r': V/V' \to T$ and  $r'':V' \to W$ are I-graded linear isomorphisms.

The maps are defined as follows. 
$q_3: (x,V') \mapsto x$, $q_2: (x, V', r', r'') \mapsto (x, V')$
and $q_1: (x, V', r', r'') \mapsto (x', x'')$ with 
$x'_{\omega}=r'_{t(\omega)}\,  (x_{V/V'})_{\omega} \, (r')^{-1}_{s(\omega)}$ and 
$x''_{\omega}=r''_{t(\omega)}\, (x_{V'})_{\omega} \,
(r'')^{-1}_{s(\omega)}$  for all $\omega\in \Omega$.

It is well-known that $q_3$ is proper,
$q_2$ is principal $G_T\times G_W$ bundle of fiber
dimension $d_2=\sum_{i\in I} (|T_i|^2+ |W_i|^2)$
and $q_1$ is smooth with connected fibers of fiber dimension 
$d_1=\sum_{i\in I} (|T_i|^2+|W_i|^2) +\sum_{\omega\in \Omega}
|T_{s(\omega)}|\,|W_{t(\omega)}| +\sum_{i\in I} |T_i|\,|W_i|$.
Here we use the notation
\[
|V|=\dim V.
\]
From the diagram in this section and the above properties, 
we have a functor 
\[
(q_{3})_!\; (q_{2})_{\flat}\; q_1^*: \mathcal Q_T\boxtimes \mathcal
Q_W \to \mathcal D(E_V).
\]
Given any $K_1\in \mathcal Q_T$ and $K_2 \in \mathcal Q_W$ (see ~\ref{simple}), we set
\[
K_1\circ K_2= (q_{3})_!\; (q_{2})_{\flat}\; q_1^* (K_1 \boxtimes K_2) \, [d_1-d_2].
\]

\begin{lem}
\label{ind1}
\begin{enumerate}

\item $K_1 \circ K_2 \in \mathcal Q_V$.

\item $K_1 \circ K_2 \in \mathcal S_V$ if $K_1\in \mathcal S_T$ and 
      $K_2\in \mathcal S_W$ (see ~\ref{semisimple}).
 
\item $L_{\mbf{i',a'}} \circ L_{\mbf{i'',a''}}= L_{\mbf{i'i'',a'a''}}$, 
      where $(\mbf{i',a'})\in \mathcal X_{|T|}$, 
      $(\mbf{i'',a''}) \in \mathcal X_{|W|}$, and $(\mbf{i'i'',a'a''}) \in
      \mathcal X_{|V|}$ is the concatenation of the two sequences.
\end{enumerate}
\end{lem}

Lemma ~\ref{ind1} (1) and (3) are proved in ~\cite{lusztig2}. Lemma
~\ref{ind1} (2) is proved in ~\cite{KS}, but see Lemma ~\ref{LemmaA}.

\subsection{Compatibility of induction functors on 
$E_V$, $\mathcal E_{\mbf{i,a}}$ and $\mathcal E_{\mbf{i,a}}^+$}
\label{inductionB}

In this section, we assume that $(\mbf  i, \mbf a)$ is the concatenation of the pairs $(\mbf i', \mbf a')$ and $(\mbf i'', \mbf a'')$.
Consider the following diagram
\[
\label{A}
\begin{CD}
\mathcal E_{\mbf{i',a'}}\times \mathcal E_{\mbf{i'',a''}} @<r_1<<
\mathcal E' @>r_2>> \mathcal E'' @>r_3>> \mathcal E_{\mbf{i,a}}\\
@V\psi_{\mbf{i',a'}}\times\psi_{\mbf{i'',a''}}VV @V\psi'VV @V\psi''VV @V\psi_{\mbf{i,a}}VV\\
E_T\times E_W @<q_1<< E' @>q_2>> E'' @>q_3>> E_V
\end{CD}
\tag{A}
\]
where the bottom row is the diagram defined in 
Section ~\ref{inductionA}, $\mathcal E_{\mbf{i,a}}$ is defined in 
Section  ~\ref{factorization}.
\[
\mathcal E''= \{ (x, \mbf D^{\bullet}; V')\; |\; 
(x,\mbf D^{\bullet})\in \mathcal E_{\mbf{i,a}},  x(V')\subseteq V',
V'\in \mbf D^{\bullet}, \dim V'=\dim W\},
\] 
where $V'$ is an $I$-graded subspace of $V$, and $V'\in \mbf D^{\bullet}$ means that 
$D^l\supseteq V'\cap V^- \supseteq D^{l+1}$ for some $l$ and 
\[
\mathcal E'= \{ (x, \mbf D^{\bullet}; V'; r', r'')\; |\; 
(x,\mbf D^{\bullet}; V' )\in \mathcal E'',  r': V/V' \tilde{\to} T,
r'':V' \tilde{\to} W \}.
\] 
The vertical morphisms are the obvious projections and 
\begin{align*}
&r_3: (x, \mbf D^{\bullet}; V') \mapsto (x,\mbf D^{\bullet}),\quad
r_2: (x,\mbf D^{\bullet}; V'; r',r'') \mapsto (x, \mbf D^{\bullet}; V'),
\; \mbox{and}\\
&r_1: (x, \mbf D^{\bullet}; V'; r',r'') \mapsto 
((x', \mbf D'^{\bullet}), (x'', \mbf D''^{\bullet})),
\end{align*}
where $(x',x'')$ is defined in Section ~\ref{inductionA},  
$\mbf D''^{\bullet}=r'' (\mbf D^{\bullet} \cap V')$,
and
$\mbf D'^{\bullet}=r'(p_T( \mbf D^{\bullet}))$.
(Note that the imaginary parts of the pairs $(\mbf{i',a'})$ and $(\mbf{i'',a''})$ are
uniquely determined by the types of the flags $p_T( \mbf D^{\bullet})$ and
$\mbf D^{\bullet} \cap V'$, respectively, although 
$(\mbf{i',a'})$ and $(\mbf{i'',a''})$ are not.)

\begin{lem} [\cite{KS}]
\label{LemmaA}
For any $R_1\in \widehat{\mathcal E}_{\mbf{i',a'}}$ and 
$R_2\in \widehat{\mathcal E}_{\mbf{i'',a''}}$ (see Section ~\ref{semisimple}), we have 
\begin{align*}
(\psi_{\mbf{i',a'}})_! (R_1) \circ (\psi_{\mbf{i'',a''}})_!( R_2)=
(\psi_{\mbf{i,a}})_!(r_3)_! (r_2)_{\flat} r_1^*(R_1\boxtimes R_2)[d_1-d_2].
\end{align*}
\end{lem}

\begin{proof}
Observe that all squares in the diagram above are commutative. Moreover,
the left and the middle squares are cartesian. By base change for proper morphism, we have
\begin{align*}
(\psi_{\mbf{i',a'}})_! (R_1) \circ (\psi_{\mbf{i'',a''}})_!( R_2)
&=(q_3)_! (q_2)_{\flat} (q_1)^* (\psi_{\mbf{i',a'}} \times
\psi_{\mbf{i'',a''}})_!(R_1\boxtimes R_2) [d_1-d_2]\\
&=(q_3)_! (q_2)_{\flat} (\psi')_! (r_1)^*(R_1\boxtimes R_2) [d_1-d_2]\\
&=(q_3)_! (\psi'')_! (r_2)_{\flat}(r_1)^*(R_1\boxtimes R_2) [d_1-d_2]\\
&=(\psi_{\mbf{i,a}})_! (r_3)_!(r_2)_{\flat}(r_1)^*(R_1\boxtimes R_2) [d_1-d_2].
\end{align*}
Lemma follows.
\end{proof}
We set 
$R_1\circ R_2=(r_3)_! (r_2)_{\flat} r_1^*(R_1\boxtimes R_2)[d_1-d_2]$. Then the
identity in Lemma above reads
\begin{equation}
\label{equationa}
(\psi_{\mbf{i',a'}})_! (R_1) \circ (\psi_{\mbf{i'',a''}})_!( R_2)=
(\psi_{\mbf{i,a}})_!(R_1\circ R_2).
\end{equation}

Similarly, consider the diagram
\[
\label{B}
\begin{CD}
\mathcal E_{\mbf{i',a'}}\times \mathcal E_{\mbf{i'',a''}} @<r_1<<
\mathcal E' @>r_2>> \mathcal E'' @>r_3>> \mathcal E_{\mbf{i,a}}\\
@V\delta_{\mbf{i',a'}}\times\delta_{\mbf{i'',a''}}VV @VVV @VVV @V\delta_{\mbf{i,a}}VV\\
\mathcal E_{\mbf{i',a'}}^+\times \mathcal E_{\mbf{i'',a''}}^+ @<r_1^+<<
\mathcal E'^+ @>r_2^+>> \mathcal E''^+ @>r_3^+>> \mathcal
E_{\mbf{i,a}}^+,
\end{CD}
\tag{B}
\]
where the top row is the top row in Diagram ~\ref{A}, $\mathcal
E^+_{\mbf{i,a}}$ is defined in Section ~\ref{relation}, the objects
$\mathcal E'^+$ and $\mathcal E''^+$ are defined in the same way as
the varieties $\mathcal E'$ and $\mathcal E''$
with $x\in E_V$ replaced by $x\in E_{V,\Omega^+}$, and the vertical
maps are projections. 

Observe that the squares in Diagram ~\ref{B} are commutative and the
right and middle squares are cartesian. 
By using a similar argument in the Proof of Lemma ~\ref{LemmaA}, one has
\begin{align*}
\delta^*_{\mbf{i',a'}}(S_1^+) \circ \delta^*_{\mbf{i'',a''}}(S_2^+)=
\delta^*_{\mbf{i,a}}(r_3^+)_! (r_2^+)_{\flat} (r_1^+)^*(S_1^+ \boxtimes S_2^+)[d_1-d_2],
\end{align*}
for any semisimple complex 
$S_1^+ \in \mathcal D(\mathcal E^+_{\mbf{i',a'}})$ and 
$S_2^+\in \mathcal D(\mathcal E^+_{\mbf{i'',a''}})$, 
which are, up to shifts, direct summands of $(\phi_{\mbf{i',a'}}^+)_!(\bar{\mathbb Q}_l)$ and 
$(\phi_{\mbf{i'',a''}}^+)_!(\bar{\mathbb Q}_l)$ respectively. 
For simplicity, we set 
$S_1^+\circ S_2^+=(r_3^+)_! (r_2^+)_{\flat} (r_1^+)^*(S_1^+ \boxtimes
S_2^+)[d_1^+-d_2]$, where $d^+_1$ is the fibre dimension of $r_1^+$.
(Note that the fibre dimension of $r_1$ and $r_1^+$ are not the same
in general.)
Then the above identity  reads
\begin{align}
\label{equationb}
\delta^*_{\mbf{i',a'}}(S_1^+) \circ \delta^*_{\mbf{i'',a''}}(S_2^+)=
\delta^*_{\mbf{i,a}}(S_1^+\circ S_2^+)[d_1^+ - d_1].
\end{align}
Let $d_{\mbf{i,a}}$ be the fibre dimension of
$\delta_{\mbf{i,a}}$. By direct computation, 
\[
d_{\mbf{i,a}}=d_1-d_1^++d_{\mbf{i',a'}}+d_{\mbf{i'',a''}}.
\]
From this and equation (\ref{equationb}), one has

\begin{lem}
\label{LemmaB}
For any semisimple complex $S_1^+ \in \mathcal D(\mathcal
E^+_{\mbf{i',a'}})$ and $S_2^+\in \mathcal D(\mathcal
E^+_{\mbf{i'',a''}})$ which are shifts of direct summands of 
$(\phi_{\mbf{i',a'}})_!(\bar{\mathbb Q}_l)$ and
$(\phi_{\mbf{i'',a''}})_!(\bar{\mathbb Q}_l)$ respectively,
\begin{align}
\delta^*_{\mbf{i',a'}}[d_{\mbf{i',a'}}](S_1^+) \circ 
\delta^*_{\mbf{i'',a''}}[d_{\mbf{i'',a''}}](S_2^+)=
\delta^*_{\mbf{i,a}}[d_{\mbf{i,a}}] (S_1^+\circ S_2^+).
\end{align}
\end{lem}

Finally, consider the following commutative diagram
\[
\begin{CD}
\label{C}
\mathcal E_{\mbf j'}^+\times \mathcal E_{\mbf j''}^+ @<r_{\mbf j, 1}^+<<
\mathcal E_{\mbf j}'^+ @>r_{\mbf j, 2}^+>> \mathcal E_{\mbf j}''^+
@>r_{\mbf j, 3}^+>> \mathcal E_{\mbf j}^+,\\
@V\beta_{\mbf{i',a'}}\times \beta_{\mbf{i'',a''}}VV 
@VVV @VVV @V \beta_{\mbf{i,a}}VV\\
\mathcal E_{\mbf{i',a'}}^+\times \mathcal E_{\mbf{i'',a''}}^+ @<r_1^+<<
\mathcal E'^+ @>r_2^+>> \mathcal E''^+ @>r_3^+>> \mathcal
E_{\mbf{i,a}}^+,
\end{CD}
\tag{C}
\]
where 
the sequence $\mbf j$ is defined in Section ~\ref{bridge} with respect
to the sequence $(\mbf{i,a})$, the bottom row is defined in Diagram (B),
and the top row is defined the same as the bottom row with the pair
$(\mbf{i,a})$ replaced by $\mbf j$. The vertical maps are obvious
projections. To avoid confusion, we put a subscript $\mbf j$ on the
morphisms and the varieties in the middle in the top row. 
Note that again all squares are commutative, and the left and middle
squares are cartesian.

\begin{lem}
\label{LemmaC}
For any semisimple complex 
$R_1^+ \in \mathcal D(\mathcal E^+_{\mbf{i',a'}})$ and
$R_2^+ \in \mathcal D(\mathcal E^+_{\mbf{i'',a''}})$
which are shifts of direct summands of 
$(\phi_{\mbf{i',a'}})_!(\bar{\mathbb Q}_l)$ and
$(\phi_{\mbf{i'',a''}})_!(\bar{\mathbb Q}_l)$ respectively,, 
\begin{align}
\label{equationc}
(\beta_{\mbf{i,a}})_!(R_1^+\circ R_2^+)
=(\beta_{\mbf{i',a'}})_!(R_1^+)\circ
(\beta_{\mbf{i'',a''}})_!(R_2^+).
\end{align}
\end{lem}

The proof goes exactly the same as the proof of Lemma
~\ref{LemmaA}.

\section{Based algebras}

\subsection{Lusztig's algebras}
\label{algebraK}
Let $\mathcal{K}_V=\mathcal K(\mathcal Q_V)$ be the Grothendieck group of the category
$\mathcal Q_V$, i.e.,
it is the abelian group with one generator $\langle L\rangle$ for each isomorphism class
of objects in $\mathcal{Q}_V$ with relations: 
$\langle L\rangle+\langle L'\rangle=\langle L''\rangle$ 
if $ L'' \cong L\oplus L'$.
  
Recall that  $\mathbb A=\mathbb Z[v,v^{-1}]$.
$\mathcal K_V$ has an $\mathbb A$-module structure defined by 
$v^n\langle L\rangle =\langle L[n]\rangle $ 
for any generator $\langle L\rangle\in\mathcal{Q}_V$ and 
$n\in \mathbb Z$.
Observe that $\mathcal K_V$ is a free $\mathbb A$-module with basis $\langle L\rangle$
where $\langle L\rangle$ runs over $\mathcal P_V$.
Moreover, $\mathcal{K}_V \cong \mathcal{K}_{V'}$,
for any $V$ and $V'$ such that $|V|=|V'|$. 

For each $\nu\in \mathbb N[I]$, we fix an $I$-graded vector space $V$ of dimension $\nu$.
Let 
\begin{align*}
&\mathcal{K}_{\nu}=\mathcal{K}_V,\quad 
\mathcal{K}=\oplus_{\nu\in \mathbb{N}[I]}\mathcal{K}_{\nu},\quad
\mathcal{K}_Q=\mathbb{Q}(v)\otimes_{\mathbb A}\mathcal{K},\\
&\mathcal P_{\nu}=\mathcal P_V
\quad \text{and} \quad
\mathcal P=\sqcup_{\nu\in \mathbb N[I]} \mathcal P_{\nu}.
\end{align*}
For any $\alpha, \beta\in \mathbb N[I]$, the functor $\circ$ defined
in Section ~\ref{inductionA} induces an 
$\mathbb A$-linear map
\[
\circ: \mathcal K_{\alpha} \otimes_{\mathbb A} \mathcal K_{\beta} 
\to \mathcal K_{\alpha +\beta}.
\]
By adding up these linear maps, we have an $\A$-linear map
\[
\circ: \mathcal K \otimes_{\mathbb A} \mathcal K \to \mathcal K.
\] 
Similarly, the operation $\circ$ induces a $\mathbb Q(v)$-linear map
$\circ:\mathcal K_Q \otimes_{\mathbb Q(v)} \mathcal K_Q \to \mathcal K_Q.$

\begin{prop}
\label{algebraprop}
\begin{enumerate}
\item $(\mathcal K, \circ)$ is an associative algebra 
      over $\mathbb A$, while  $(\mathcal K_Q, \circ)$ is a $\mathbb Q(v)$-algebra.
\item $\mathcal P$ is an $\mathbb A$-basis of $(\mathcal K, \circ)$ 
      and a $\mathbb Q(v)$-basis of $(\mathcal K_Q, \circ)$. 
\end{enumerate}
\end{prop}

See ~\cite{lusztig2} or ~\cite{lusztig4} for a proof of
associativity. 

Following Lusztig, we call $\mathcal P$ the $canonical$ $basis$ of the algebras
$(\mathcal K,\circ)$ and $(\mathcal K_Q,\circ)$.
From now on, we simply write $\mathcal K$ (resp. $\mathcal K_Q$) for
the algebra $(\mathcal K, \circ)$ (resp. $(\mathcal K_Q, \circ)$).

\subsection{The algebras $\mathcal M$ and $\mathcal M_Q$}
\label{algebraM}
Similar to the algebras $\mathcal K$ and $\mathcal K_Q$, 
we define algebras $\mathcal M$ and $\mathcal M_Q$ as follows.

Let $\mathcal M_V$ be the $\A$-submodule of $\mathcal K_V$ generated
by elements in $\mathcal R_V$ (see ~\ref{semisimple}).
For each $\nu\in \mathbb N[I]$, we fix an $I$-graded vector space $V$ of dimension $\nu$.
Let 
\begin{align*}
&\mathcal M_{\nu}=\mathcal M_V,\quad 
\mathcal M=\oplus_{\nu\in \mathbb{N}[I]}\mathcal M_{\nu},
\quad 
\mathcal M_Q=\mathbb Q(v)\otimes_{\A}\mathcal M,\\
&\mathcal R_{\nu}=\mathcal R_V
\quad \text{and} \quad
\mathcal R=\sqcup_{\nu\in \mathbb N[I]} \mathcal R_{\nu}.
\end{align*}
By Lemma ~\ref{ind1} (2),  the functor $\circ$ defined
in Section ~\ref{inductionA} induces an $\mathbb A$-linear map
\[
\circ: \mathcal M_{\alpha} \otimes_{\mathbb A} \mathcal M_{\beta} 
\to \mathcal M_{\alpha +\beta},
\quad 
\mbox{for}
\;
\alpha, \beta \in \mathbb N[I].
\]
By adding up these linear maps, we have linear maps
\[
\circ: \mathcal M \otimes_{\mathbb A} \mathcal M \to \mathcal M
\quad
\mbox{and}
\quad
\circ: \mathcal M_Q \otimes_{\mathbb Q(v)} \mathcal M_Q \to \mathcal M_Q,
\] 
where the first map is $\A$-linear while the second map is $\mathbb
Q(v)$-linear.
 
\begin{prop}
\label{M}
\begin{enumerate}
\item The pair $(\mathcal M, \circ)$ is an $\A$-subalgebra of $\mathcal
      K$, and $(\mathcal M_Q,\circ)$ is a $\mathbb Q(v)$-subalgebra of
      $\mathcal K_Q$.
\item $\mathcal R$ is an $\A$-basis of $(\mathcal M,\circ)$, and a
      $\mathbb Q(v)$-basis of $(\mathcal M_Q,\circ)$.
\item $(\mathcal M, \circ)=(\mathcal K,\circ)$ and $\mathcal
      R=\mathcal P$ if $l_i\geq 2$ for all
      $i\in I^-$ (see ~\ref{quivers}).
\end{enumerate}
\end{prop}

Proposition ~\ref{M} (1) is by definition. The proof of Proposition ~\ref{M} (2) follows
essentially word by word from the proof of Proposition 4.3 in
~\cite{KS}.  We leave it to the reader. Proposition ~\ref{M} (3) is
due to Lemma ~\ref{perverse} (2).

We call $\mathcal R$ the $canonical$
$basis$ of $\mathcal M$ and $\mathcal M_Q$.

\subsection{}
\label{newalgebra}
By replacing  the pair $(\mathcal R_V, \mathcal S_V)$ with the pair 
$(\mathscr R_V, \mathscr S_V)$ (\ref{certain}) and following the
construction in Section ~\ref{algebraM}, we obtain similar based algebras,
denoted by $\mathscr M=(\mathscr M, \circ)$ and 
$\mathscr M_Q=(\mathscr M_Q,\circ)$, respectively, with basis
$\mathscr R=\sqcup_{\nu\in \mathbb N[I]} \mathscr R_{\nu}$.

By replacing  the pair $(\mathcal P_V, \mathcal Q_V)$ with the pair 
$(\widehat{\mathscr E}_{\nu}, \widehat{\mathscr F}_{\nu})$
(\ref{certain}),
the diagram in Section ~\ref{inductionA} with the first row
in Diagram (\ref{A}) in Section ~\ref{inductionB}, and 
following the
construction in Section ~\ref{algebraK}, we obtain similar based algebras,
denoted by $\mathscr K=(\mathscr K, \circ)$ and 
$\mathscr K_Q=(\mathscr K_Q,\circ)$, respectively, with basis
$\widehat{\mathscr E}=\sqcup_{\nu\in \mathbb N[I]} \widehat{\mathcal E}_{\nu}$.

Finally, the data $(\widehat{\mathscr E}_{\nu}, \widehat{\mathscr F}_{\nu})$
(\ref{certain}) and the second row in Diagram (\ref{B}) give rise to a
based algebra, denoted by 
$\mathscr K^+=(\mathscr K^+, \circ)$ and 
$\mathscr K^+_Q=(\mathscr K^+_Q,\circ)$, respectively, with basis
$\widehat{\mathscr E}^+=\sqcup_{\nu\in \mathbb N[I]} \widehat{\mathcal E}^+_{\nu}$.

The algebras $\mathscr M$, $\mathscr K$ and $\mathscr K^+$ are related
as follows. The bijective map
\[
\psi: \widehat{\mathscr E}_{\nu} \to \mathscr R_V
\quad
\widehat{R} \mapsto (\psi_{\nu})_!(\widehat{R})
\quad \forall
\widehat{R}\in \widehat{\mathscr E}_{\nu}
\]
extends to an isomorphism of $\A$-modules
$\psi: \mathscr K \tilde{\longrightarrow} \mathscr M$.
The bijective map
\[
\delta_{\nu}^*[d_{\nu}]: \widehat{\mathscr E}_{\nu}^+ \to
\widehat{\mathscr E}_{\nu}
\quad
\widehat R^+ \mapsto \delta_{\nu}^* [d_{\nu}](\widehat R^+)
\quad \forall \widehat R^+\in \widehat{\mathscr E}_{\nu}^+
\]
($d_{\nu}$ is the fibre dimension of $\delta_{\nu}$) defines an isomorphism of $\A$-modules
$\delta: \mathscr K^+ \tilde{\longrightarrow} \mathscr K$.

Now equation (1) in Section \ref{inductionB} implies that $\psi$ is an $\A$-algebra
homomorphism, while Lemma ~\ref{LemmaB} implies that $\delta$ is an
$\A$-algebra homomorphism.  In summary,

\begin{prop} 
\label{variant}
The maps 
$
\mathscr K^+ \overset{\delta}{\longrightarrow}
\mathscr K \overset{\psi}{\longrightarrow} \mathscr M 
$
are isomorphisms of $\A$-algebras. Moreover,
$\delta(\widehat{\mathscr E}^+)= 
\widehat{\mathscr E}$ and
$\psi(\widehat{\mathscr E})= \mathscr R$.
\end{prop}

\begin{rem}
The pair  $(\mathscr M, \mathscr R)$ first appeared in
~\cite{KS}. 
\end{rem}

Let $C=C(Q)$ be the Borcherds-Cartan matrix associated to the quiver $Q$
(\ref{quivers}). Let $\U^-$ be the negative part of the quantum
generalized Kac-Moody algebra associated to $C$
(\ref{quantum}). Kang and Schiffmann showed that 

\begin{thm}[\cite{KS}]
\label{KSthm}
\begin{enumerate}
\item The assignment $L_{i, 1} \mapsto F_i$ defines an $\A$-algebra isomorphism
$\mathscr M \overset{\phi}{\to}  \, _{\mathbb A}\U^-$, and a $\mathbb Q(v)$-algebra isomorphism $\mathscr M_Q \to \U^-$. 
Here $L_{i,1}$ is the unique element in $\mathscr R_i$ for any $i\in I$.

\item $\mbf B=\phi(\mathscr R)$ is the canonical basis of
      $_{\mathbb A}\U^-$ and $\U^-$. Moreover, $\mbf B$ coincides with the global
      crystal base of $\U^-$ in \cite{JKK} if $l_i\geq 2$
      $(i\in I^-)$.

\item For any $R\in \mathscr R_{\nu}$ with  $n=n_i(R)>0$ ($i$ is a
sink), 
$R=R_1\circ R_2 +R_3$
where $R_1\in \mathscr R_{ni}$, $R_2\in \mathscr R_{\nu-ni}$ and
$R_3\in \mathscr M_{\nu}$ with $n_i(R_3)>n$. (see ~\ref{proof1} for
notations.)
\end{enumerate}
\end{thm}

In view of Proposition ~\ref{variant} and Theorem ~\ref{KSthm}, we have

\begin{cor}
\label{variantcorollary}
\begin{enumerate}
\item The algebras $\mathscr K^+_Q$ and  $\mathscr K_Q$ are geometric
      realizations of $\U^-$ via the maps $\delta$, $\psi$ and $\phi$.
\item $\widehat{\mathscr E}^+$ and $\widehat{\mathscr E}$ are
      realizations of the canonical basis $\mbf B$ of $\U^-$. 
\item For any $S^+ \in \widehat{\mathscr E}^+_{\nu}$ such that 
      $n=n_i(S^+)>0$ ($i$ a sink), $S^+=S^+_1\circ S^+_2 +S_3^+$ for
      some $S^+_1\in \widehat{\mathscr E}^+_{mi}$ ($n\geq m>0$),
      $S^+_2\in \widehat{\mathscr E}^+_{\nu-mi}$, and $S_3^+\in
      \mathscr K^+_{\nu}$ with $n_i(S_3^+)>n$.

\end{enumerate}
\end{cor}

Assume that $C=(c_{ij})_{i,j\in I}$ and $D=(d_{ij})_{i, j\in I}$ are
two Borcherds-Cartan matrices.
We call $C \approx D$ if
\[
c_{ij}=d_{ij} \quad \forall i\neq j \in I
\]
and 
\[
c_{ii}=2 \quad \mbox{if and only if} \quad d_{ii}=2.
\]
In other words, $C\approx D$ if $C$ and $D$ coincide at all entries except the
imaginary diagonal entries.
Although $\U(C)$ is not isomorphic to $\U(D)$, we have (see ~\ref{quantum})
\[
\U^-(C)= \U^-(D)
\quad
\mbox{if}
\quad 
C \approx D.
\]

Let $\mbf B(C)$ be the canonical basis of $\U^-(C)$ and $\mbf B(D)$ be
the canonical basis of $\U^-(D)$. Although $\U^-(C)=\U^-(D)$, the
constructions of $\mbf B(C)$ and $\mbf B(D)$ are different. But in
view of Corollary ~\ref{variantcorollary} (2), we have

\begin{thm}
\label{C=D}
$\mbf B(C)=\mbf B(D)$, provided that $\U^-(C)=\U^-(D)$.
\end{thm}

This is because they all equal the image of 
$\widehat{\mathscr E}^+$ under the composition $\phi\psi\delta$ of morphisms.
Note that $\phi\psi\delta$ is independent of the number of loop arrows
on the imaginary vertices, if the quivers for $C$ and $D$
are chosen such that the subquivers containing all non loop arrows are
the same.

\begin{rem}
As the referee pointed out, the natural scalar products of the various algebras $\mathscr K, \mathscr K^+$ and $\mathscr M$ are different, 
despite the fact that they are isomorphic.  
This difference explains the fact that only positive halves of the various quantum generalized Kac-Moody algebras are isomorphic to each other.
\end{rem}

\subsection{Independence of orientations}
Let $Q'=(I,\Omega, s',t')$ be a quiver 
with the sets $I$ and $\Omega$ the same as $Q=(I,\Omega,s,t)$ and
$\{s'(\omega),t'(\omega)\}=\{s(\omega),t(\omega)\}$
for all $\omega \in \Omega$. Associated to $Q'$ similar algebras
defined in Section ~\ref{algebraK}, ~\ref{algebraM} and ~\ref{newalgebra}.

\begin{prop}
\label{orientation}
The algebras defined in Section ~\ref{algebraK}, ~\ref{algebraM} and ~\ref{newalgebra}
are independent of changes of  orientations, via Fourier-Deligne
transform. Moreover, the Fourier-Deligne transform preserves the
various bases under the changes of orientations
\end{prop}

The independence of the algebras $\mathcal K$ under changes of
orientations have been investigated by Lusztig in ~\cite{lusztig2}.
Notice the fact that Fourier-Deligne transform commutes with base
change (\cite{Laumon}). 
One can adapt the proof given in ~\cite{lusztig2} to the rest of the
algebras.  See ~\cite[Section 2]{KSaddendum}  for a concrete proof.

\section{A characterization of the  algebra $\mathcal M_Q$}

\subsection{}
\label{mainresults}
When the pair $(\mbf{i,a})=(i, n)$ for $i\in I$ and $n\in \mathbb N$, 
the isomorphism class of  the complex $L_{i,n}$ (see ~\ref{simple})
is in $\mathcal R_{ni}$. (Note that the $L_{i,n}$'s are simple
perverse sheaves supported on $\{0\}\subseteq E_V$ ($|V|=ni$).)

\begin{thm}
\label{theorem1}
The algebra $\mathcal M_Q$ (see ~\ref{algebraM}) is generated by the
complex $L_{i,n}$ for $i\in I$ and $n\in \mathbb N$. In particular, 
$\mathcal K_Q=\mathcal M_Q$ is generated by the $L_{i,n}$'s if $l_i\geq 2$
for all $i\in I^-$ (see ~\ref{quivers}).
\end{thm}

The proof will be given in Section ~\ref{proof1}.

Let $C=(c_{ij})_{i,j\in I}$ be the Borcherds-Cartan matrix associated
to $Q$  (see ~\ref{quivers}). We set
\[
\mbf I=I^+ \sqcup \{ni| n\in \mathbb N, i\in I^-\}, 
\]
and define a new Borcherds-Cartan matrix 
$\widetilde C=(\tilde c_{\mbf{ij}})_{\mbf{i,j\in I}}$
by
\[
\tilde c_{\mbf{ij}}=(\mbf i, \mbf j)
\]
where $(-,-)$ is the bilinear form associated to $Q$ (see
~\ref{quivers}).

\begin{prop}
\label{prop}
There exists elements $\xi_{\mbf i}\in \mathcal M_{\mbf i}$ for
$\mbf{i\in I}$ ($\xi_{\mbf i}=L_i$ if $i\in I$) such that the assignment 
$\xi_{\mbf i} \mapsto F_{\mbf  i}$ defines a $\mathbb Q(v)$-algebra
isomorphism
\[
\mathcal M_Q \cong \mbf U^-(\widetilde C),
\]
where $\mbf U^-(\widetilde C)$ is the negative part of the quantum generalized
Kac-Moody algebras associated to the matrix $\widetilde C$ (see ~\ref{quantum}).
\end{prop}

\begin{proof}
Recall that 
$\mathcal M_Q=\mathbb Q(v) \otimes_{\A} \mathcal M$. 
$\mathcal M_Q$ then has a natural $\mathbb N[I]$-grading inherited
from $\mathcal M$, i.e., 
$\mathcal M_Q=\oplus_{\nu\in \mathbb N[I]} \mathcal M_Q(\nu)$, 
where 
$\mathcal M_Q(\nu)=\mathbb Q(v) \otimes_{\A} \mathcal M_{\nu}$.

Define a $\mathbb Q(v)$-algebra structure on $\mathcal M_Q \otimes \mathcal M_Q$
by
\[
(x\otimes y) (x'\otimes y')=v^{(\mbox{deg}(y), \mbox{deg}(x'))}
xx'\otimes yy',
\]
where $x$ $y$, $x'$ and $y'$ are homogeneous elements, 
$\mbox{deg}(y)$ and $\mbox{deg}(x')$ are the degree of $y$ and $x'$, respectively.

Following Lusztig (\cite{lusztig4}), we can associate to $\mathcal M_Q$
an algebra homomorphism
\[
r: \mathcal M_Q \to \mathcal M_Q\otimes \mathcal M_Q
\]
and a symmetric  nondegenerate linear form
\[
(-, -) : \mathcal M_Q \otimes \mathcal M_Q \to \mathbb Q(v),
\]
satisfying 
\begin{enumerate}
\item $(1,1)=1$, and $(\mathcal M_Q(\nu), \mathcal M_Q (\nu'))=0$ if $\nu \neq \nu'$.

\item $(x,yy')=(r(x), y\otimes y')$ for $x, y, y' \in \mathcal M_Q$.
\end{enumerate}

Given any $i\in I^-$, we set
\[
\mathcal H_{mi}=\left \{ x \in \mathcal M_Q\;|\;  \left ( x, \sum_{m',m''}
\mathcal M_Q(m'i) \mathcal M_Q(m''i) \right )=0 \right \},
\]
where the sum runs over all pairs $(m',m'')$ such that $m'+m''=m$ and
$0 < m', m''<m$. Note that the elements of the following form span
$\mathcal M_Q(mi)$
\[
L_{i,m_1} \circ L_{i,m_2} \circ \cdots \circ L_{i,m_l},
\quad m_1+ \cdots +m_l=m.
\]
Note also that all these elements
are in $\sum_{m',m''}\mathcal M_Q(m'i)\mathcal M_Q(m''i)$, except
$L_{m, i}$. 
From these observations and the nondegeneracy of the linear form on
$\mathcal M_Q$, we have $\dim_{\mathbb Q(v)} \mathcal H_{mi}=1$.
For a nonzero element $x$ in $\mathcal H_{mi}$, $x$ must be of the
form $x=a_m F_{mi}+ x'$ where $a_m\in \mathbb Q(v)$ and 
$x'\in  \sum_{m',m''}\mathcal M_Q(m'i)\mathcal M_Q(m''i)$. We set
$\xi_{mi}=a_m^{-1}x$, for $m\in\mathbb N$ and $i\in I^-$.
Note that $\xi_{i}=L_{i,1}$ for all $i\in I^-$. Let $\xi_{i}=L_{i,1}$
if $i\in I^+$. So the set $\{\xi_{\mbf i}\;|\;\mbf{i\in I}\}$ generates $\mathcal M_Q$.

Now Proposition ~\ref{prop} follows from the argument in ~\cite{SV}. 
\end{proof}

\subsection{Proof of Theorem ~\ref{theorem1}}
\label{proof1}
We need some preparations. 
Let $Path(i)$ be the set of all paths $p= \omega_1\cdots \omega_n$ for $\omega_1,\cdots,\omega_n \in \Omega(i)$ and $n\in \mathbb Z_{>0}$.
Given any $x\in E_V$, we set 
\begin{align*}
V_{i,x}=\sum_{\omega\in\Omega^+: t(\omega)=i} x_{\omega}(V_{s(\omega)}),
\quad \mbox{and}\quad
V_i(x)=V_{i,x}+ \sum_{p\in Path(i)} x_p (V_{i,x}).
\end{align*} 
(See Section ~\ref{quivers} for notations.)
A vertex $i$ in $Q$ is called a $sink$ if $t(\omega)\neq i$ for all
$\omega\in \Omega^+$. Given any $x\in E_V$, we set
\[
n_i(x)=\mrm{codim}_{V_i} V_i(x) \quad \mbox{if}\; i \; \mbox{is a sink}.
\]
More generally, for any semisimple complex $P \in \mathcal D(E_V)$, 
we set
\[
n_i(P)=\mrm{min} \{ n_i(x)\;|\; x\in \mrm{supp}(P)\},
\]
where $\mrm{supp}(P)$ is the support of $P$.

For any semisimple perverse sheaf $(\psi_{\mbf{i,a}})_!(R) \in \mathcal R_V$
with $\mbf i=(i_1,\cdots,i_m)$, we may assume
that $i=i_1$ is a sink via change of orientations in view of Proposition
~\ref{orientation}. Under this assumption, one see that
\[
N=n_i ((\psi_{\mbf{i,a}})_!(R)) >0.
\]

On $\mathbb N[I]$, we define a partial order $<$ by $\tau<\nu$ if
$\tau_j\leq \nu_j$ for all $j \in I$ and $\tau_{j_0}<\nu_{j_0}$ for
some $j_0\in I$.

Now we begin to prove Theorem ~\ref{theorem1}.
Let $\mathfrak M$ be the subalgebra of $\mathcal M$ generated by
$L_{(i,n)}$ for $i\in I$ and $n\in \mathbb N$.
Assume that all semisimple perverse sheaves in $\mathcal R_{\tau}$ are
in $\mathfrak M$ for
all $\tau<\nu$ and all semisimple perverse sheaves 
$P'\in \mathcal R_{\nu}$ such that $n_i(P')>N>0$ are in $\mathfrak M$.
To show Theorem ~\ref{theorem1}, it suffices to show that 
any semisimple perverse sheaf 
$P=(\psi_{\mbf{i,a}})_!(R)\in\mathcal R_{\nu}$ 
with $n_i(P)=N$ is in $\mathfrak M$.

Let us consider the diagram in Section ~\ref{relation}. By 
property (\ref{relation}.2),
\[
(\phi_{\mbf{i,a}})_!(\bar{\mathbb Q}_l)
=\delta^*_{\mbf{i,a}}(\phi^+_{\mbf{i,a}})_!(\bar{\mathbb Q}_l).
\]
Since $\delta_{\mbf{i,a}}$ is a vector bundle, $\delta_{\mbf{i,a}}^*$ is
fully faithful. So there is a unique simple perverse sheaf $R^+$ as a
summand of $(\phi^+_{\mbf{i,a}})_!(\bar{\mathbb Q}_l)$ such that
$\delta_{\mbf{i,a}}^*[d_{\mbf{i,a}}](R^+)=R$. Observe that 
$n_i(R^+)>0$ since $N>0$.

Consider the diagram in Section ~\ref{bridge}. By property
(\ref{bridge}.1), we have
\[
(\phi_{\mbf{j}}^+)_! (\bar{\mathbb  Q}_l)
=\beta_{\mbf{i,a}}^*(\phi^+_{\mbf{i,a}})_! (\bar{\mathbb Q}_l).
\]
From this and the fact that $\beta_{\mbf{i,a}}$ is smooth with
connected fibres, we see that
\[
S^+=\beta_{\mbf{i,a}}^*(R^+)
\] 
is a simple perverse sheaf (up to a shift) on 
$\mathcal E_{\mbf j}$ and a direct summand of 
$(\phi^+_{\mbf j})_!(\bar{\mathbb Q}_l)$.
Moreover $n_i(S^+)=n_i(R^+)>0$. By Corollary ~\ref{variantcorollary} (3),
\begin{align}
\label{step1}
\beta_{\mbf{i,a}}^*(R^+)=S^+_1\circ S^+_2 + S^+_3
\end{align}
where $S^+_1\in \widehat{\mathscr E}^+_{mi}$ ($n\geq m >0$) 
and $S^+_2\in \widehat{\mathscr E}^+_{\nu-mi}$,
and $S^+_3$ is a semisimple complex on $\mathcal E_{\nu}^+$ such
that $n_i(S^+_3)> n_i(S^+)$. Thus by Lemma ~\ref{LemmaC},
\begin{align}
\label{step2}
(\beta_{\mbf{i,a}})_! \beta_{\mbf{i,a}}^*(R^+)=
(\beta_{\mbf{i',a'}})_!(S^+_1)\circ 
(\beta_{\mbf{i'',a''}})_!(S^+_2) +
(\beta_{\mbf{i,a}})_!( S^+_3),
\end{align}
where $(\beta_{\mbf{i',a'}})_!(S^+_1)\in \mathcal M_{mi}$,  
$(\beta_{\mbf{i'',a''}})_!(S^+_2) \in \mathcal M_{\nu-mi}$ and
$(\beta_{\mbf{i,a}})_!( S^+_3)\in \mathcal M_{\nu}$ with
$n_i((\beta_{\mbf{i,a}})_!( S^+_3))> n_i(R^+)$.

By applying the functor $\delta_{\mbf{i,a}}^*[d_{\mbf{i,a}}]$ to equation
(\ref{step2}) and  Lemma ~\ref{LemmaB},
\begin{align}
\label{step3}
\begin{split}
&\delta_{\mbf{i,a}}^*[d_{\mbf{i,a}}] (\beta_{\mbf{i,a}})_! \beta_{\mbf{i,a}}^*(R^+)\\
&=
\delta^*_{\mbf{i',a'}} [d_{\mbf{i',a'}}] (\beta_{\mbf{i',a'}})_!(S^+_1)\circ 
\delta^*_{\mbf{i'',a''}}[d_{\mbf{i'',a''}}] (\beta_{\mbf{i'',a''}})_!(S^+_2) +
\delta^*_{\mbf{i,a}}[d_{\mbf{i,a}}](\beta_{\mbf{i,a}})_!( S^+_3).
\end{split}
\end{align}

By applying the functor $\psi_{\mbf{i,a}}$ to equation (\ref{step3}) and in view of Lemma
~\ref{LemmaC}, 
\begin{align}
\label{step4}
\begin{split}
(\psi_{\mbf{i,a}})_!&
\delta_{\mbf{i,a}}^*[d_{\mbf{i,a}}] (\beta_{\mbf{i,a}})_! \beta_{\mbf{i,a}}^*(R^+)\\
&=
(\psi_{\mbf{i',a'}})_!
\delta^*_{\mbf{i',a'}} [d_{\mbf{i',a'}}] (\beta_{\mbf{i',a'}})_!(S^+_1)
\circ 
(\psi_{\mbf{i'',a''}})_!
\delta^*_{\mbf{i'',a''}}[d_{\mbf{i'',a''}}] (\beta_{\mbf{i'',a''}})_!(S^+_2) \\
&\quad
+
(\psi_{\mbf{i,a}})_!
\delta^*_{\mbf{i,a}}[d_{\mbf{i,a}}](\beta_{\mbf{i,a}})_!( S^+_3),
\end{split}
\end{align}
By assumption, the right hand side of equation (\ref{step4}) is in
$\mathfrak M$. So is the left hand side.

Since $\beta_{\mbf{i,a}}$ is smooth with connected fibres, by
corollary 4.2.6.2 in ~\cite{BBD}, 
\[
R^+=\, ^pH^{-e_{\mbf{i,a}}} (\beta_{\mbf{i,a}})_! \beta_{\mbf{i,a}}^* (R^+),
\]
where $^pH^{-e_{\mbf{i,a}}}$ is the perverse cohomology functor at
degree $-e_{\mbf{i,a}}$ and
$e_{\mbf{i,a}}$ is the fibre dimension of $\beta_{\mbf{i,a}}$.
But $(\beta_{\mbf{i,a}})_! \beta_{\mbf{i,a}}^* (R^+)$ is semisimple, so
\[
\begin{split}
(\beta_{\mbf{i,a}})_! \beta_{\mbf{i,a}}^*(R^+)
&=\oplus_{-e_{\mbf{i,a}}\leq z\leq e_{\mbf{i,a}}} \, 
^pH^z(\beta_{\mbf{i,a}})_! \beta_{\mbf{i,a}}^* (R^+)[-z]\\
&= R^+[e_{\mbf{i,a}}]\oplus \oplus_{-e_{\mbf{i,a}}\leq z <e_{\mbf{i,a}}} R'[z].
\end{split}
\]
Thus
$(\psi_{\mbf{i,a}})_!(R)[e_{\mbf{i,a}}]$ is a direct summand of 
$(\psi_{\mbf{i,a}})_!\delta_{\mbf{i,a}}^*[d_{\mbf{i,a}}](\beta_{\mbf{i,a}})_!
\beta_{\mbf{i,a}}^* (R^+)$, as the leading term with respect to the
degree of shift. 
Moreover, for any other direct summand $P'$ in  
$(\psi_{\mbf{i,a}})_!\delta_{\mbf{i,a}}^*[d_{\mbf{i,a}}](\beta_{\mbf{i,a}})_!
\beta_{\mbf{i,a}}^* (R^+)$ is of the form
\[
P'[z]=(\psi_{\mbf{i,a}})_! (R')[z],\quad -e_{\mbf{i,a}}\leq z <e_{\mbf{i,a}}
\]
and
$n_i(P')\geq N=n_i(R^+)$.

For those $P'=(\psi_{\mbf{i,a}})_!(R')$'s such that $n_i(P')=N$, 
we can apply the same argument
again to obtain equations similar to equation (\ref{step4}): 
\begin{align}
\label{step5}
\begin{split}
(\psi_{\mbf{i,a}})_!&
\delta_{\mbf{i,a}}^*[d_{\mbf{i,a}}] (\beta_{\mbf{i,a}})_! \beta_{\mbf{i,a}}^*(R'^+)\\
=&
(\psi_{\mbf{i',a'}})_!
\delta^*_{\mbf{i',a'}} [d_{\mbf{i',a'}}] (\beta_{\mbf{i',a'}})_!(S'^+_1)
\circ 
(\psi_{\mbf{i'',a''}})_!
\delta^*_{\mbf{i'',a''}}[d_{\mbf{i'',a''}}] (\beta_{\mbf{i'',a''}})_!(S'^+_2)\\ 
&+
(\psi_{\mbf{i,a}})_!
\delta^*_{\mbf{i,a}}[d_{\mbf{i,a}}](\beta_{\mbf{i,a}})_!( S'^+_3),
\end{split}
\end{align}
Again, by induction hypothesis, 
the right hand side of (\ref{step5}) is in $\mathfrak M$, 
so is the left hand side.

Thus we have a system of finitely many linear equations indexed by
$P'$ such that $n_i(P')=N$ in $\mathfrak M$. Notice the left hand side
of each equation has a leading term of 
the form $P'[e_{\mbf{i,a}}]$. So the terms
in the left hand sides of the equations are linearly independent. Therefore
we can solve for $(\psi_{\mbf{i,a}})_!(R)$ from this system in
$\mathfrak M$. In other words, $(\psi_{\mbf{i,a}})_!(R)\in \mathfrak
M$. This finishes the proof.

\section{Comments}
{\bf 1.} 
It may be of interest to generalize the main results in this paper to the symmetrizable case.

{\bf 2.}
In view of Theorem ~\ref{C=D} and Theorem ~\ref{KSthm} (2), the global
crystal bases $\mathcal B(C)$ and $\mathcal B(D)$ in ~\cite{JKK} coincide if $C\approx D$, 
$c_{ii}\neq 0$ and $d_{ii}\neq 0$ for all $i\in I^-$. One may
conjecture that the conditions $c_{ii}\neq 0$ and $d_{ii}\neq 0$ can
be removed.

{\bf Conjecture.} $\mathcal B(C)=\mathcal B(D)$ if $\U^-(C)=\U^-(D)$. 

The conjecture in
~\cite{KS} is a consequence of the above conjecture.

{\bf 3.} 
Lusztig asked in ~\cite{lusztig3} if $\mathcal K_Q=\mathcal M_Q$. 
This is the case if the quiver $Q$ does not have any loop (\cite{lusztig2}).
When $Q$ is the Jordan quiver (the quiver with one vertex and
one loop arrow), $\mathcal K_Q=\mathcal M_Q$. This is because every
simple perverse sheaves in $\mathcal K_Q$ is the leading term of a
certain monomial. 
More generally, when
$l_i\geq 2$ for $i\in I^-$, $\mathcal K_Q=\mathcal M_Q$ by Theorem
~\ref{theorem1}. 
For $Q$ such that $l_i=1$ for some $i\in I^-$, computational evidents
show that the simple perverse sheaves in $\mathcal K_Q$ are leading
terms of some semisimple perverse sheaves
$(\psi_{\mbf{i,a}})_!(R)$. If one can show that this holds, then the
question asked by Lusztig in ~\cite{lusztig3} gets a positive answer.

\end{document}